\documentclass{article}
\usepackage{amsmath,amsthm,amscd,amssymb}
\usepackage[mathscr]{eucal}
\usepackage{amssymb}
\usepackage{latexsym}
\usepackage{amsmath}

\theoremstyle{definition}
\newtheorem{Def}{Definition}[section]
\newtheorem{Thm}[Def]{Theorem}
\newtheorem{Prop}[Def]{Proposition}
\newtheorem{Rem}[Def]{Remark}
\newtheorem{Ex}[Def]{Example}
\newtheorem{Cor}[Def]{Corollary}

\newtheorem{Lem}[Def]{Lemma}

\numberwithin{equation}{section}

\title{Congruence relations satisfied by quaternionic modular forms}

\author{Shoyu Nagaoka}

\date{}
\begin{document}

\maketitle

\begin{abstract}
The theory of quaternionic modular forms has been studied for decades as
an example of the modular forms of many variables. The purpose of this study is to provide some
congruence relations satisfied by such quaternionic modular forms.

\end{abstract}


\section{Introduction}
\label{intro}
The basis of the theory of quaternionic modular forms can be found in a study by Krieg \cite{K1}, who developed the theory. For example, he analyzed the structure of the (quaternionic) Maass
space and succeeded in obtaining an explicit formula of the Fourier coefficient of the quaternionic
Eisenstein series of degree 2.
Furthermore, he determined the structure of the graded ring of quaternionic modular forms, also 
in the case of degree 2. By contrast, except for \cite{K-N}, the $p$-adic and mod $p$ theories have not been fully studied.

In this paper, we refer to the arithmetic properties of quaternionic modular forms and provide 
the following results on the mod $p$ theory of quaternionic modular forms:
\vspace{2mm}
\\
$\bullet$\quad Ramanujan-type congruence (Theorem \ref{Ram})
\vspace{1mm}
\\
$\bullet$\quad Congruences arising from the theta operator (Theorem \ref{Theta})
\vspace{1mm}
\\
$\bullet$\quad Congruences mod 23 (Theorem \ref{mod23})
\vspace{1mm}
\\
$\bullet$\quad Congruences of Eisenstein series (Theorem \ref{CongEis})
\vspace{4mm}
\\
\section{Preliminary}
\label{preliminary}
The notations used in this study are from Krieg's study \cite{K3}. 

We denote by $\mathcal{O}$ the order of Hurwitz quaternions, i.e.,
$$
\mathcal{O}=\mathbb{Z}e_0+\mathbb{Z}e_1+\mathbb{Z}e_2+\mathbb{Z}e_3,
$$
where $e_0=\frac{1}{2}(e_1+e_2+e_3+e_4)$ and $1=e_1$,\;$e_2$,\;$e_3$,\;$e_4$ is the
standard basis for real quaternions. Let $\mathcal{O}^\sharp$ be the dual lattice, that is,
$$
\mathcal{O}^\sharp
=\mathbb{Z}2e_1+\mathbb{Z}(e_1+e_2)+\mathbb{Z}(e_1+e_3)+\mathbb{Z}(e_1+e_4).
$$
If $\tau$ denotes the reduced trace form, then the elements of the dual lattice of
$$
\text{Sym}(2;\mathcal{O})
=\{\,S\in\text{Mat}(2;\mathcal{O})\,\mid\,S={}^t\overline{S}\,\}
$$
with respect to $\tau$ is given in the following form:
$$
T=
\begin{pmatrix} n  &  \frac{1}{2}t \\ \frac{1}{2}\overline{t} & m \end{pmatrix}
\in\text{Sym}^\tau(2;\mathcal{O}),
$$
where $n,\,m\in\mathbb{Z}$,\,$t\in\mathcal{O}^\sharp$.
Given $0 \ne T\in \text{Sym}^ \tau (2; \mathcal{O})$, the ``greatest common divisor'' of
$T$ is defined by
$$
\varepsilon (T):=\text{max}\{\,d\in\mathbb{N}\,\mid\, d^{-1}T\in\text{Sym}^\tau (2;\mathcal{O})\,\}.
$$
Let $[\Gamma_2(\mathcal{O}),k]$ denote the vector space consisting of modular forms of
weight $k$ for the quaternionic modular group $\Gamma_2(\mathcal{O}):=\text{Sp}(2,\mathcal{O})$
and $[\Gamma_2(\mathcal{O}),k]_0$ denote the subspace of the cusp forms.

Each quaternionic modular form $f\in [\Gamma_2(\mathcal{O}),k]$\;$(k \equiv 0 \pmod{2})$
possesses a Fourier expansion of the form
$$
f(Z)=\sum_{\substack{T\in\text{Sym}^\tau (2;\mathcal{O}) \\ T\geq 0}}
              a(f;T)e^{2\pi i\tau (T,Z)},\quad Z\in H(2;\mathbb{H}),
$$
where $H(2;\mathbb{H})$ is the quaternionic upper-half space of degree 2.
For a subring $R\subset\mathbb{C}$, we denote by $[\Gamma_2(\mathcal{O}),k;R]$ the set
consisting of $f\in [\Gamma_2(\mathcal{O}),k]$ with $a(f;T)\in R$ for all $T$.

A modular form $f\in [\Gamma_2(\mathcal{O}),k]$ belongs to the 
{\it Maass space $\mathcal{M}(k,\mathbb{H})$ of weight $k$} if and only if a function 
$a_f^*: \mathbb{N}\cup\{ 0\}\longrightarrow \mathbb{C}$ exists such that
all $0\ne T\in\text{Sym}^\tau (2;\mathcal{O})$,\,$T\geqq 0$ satisfy
$$
a(f;T)=\sum_{\substack{d\in\mathbb{N} \\ d\mid\varepsilon (T)}}
                       d^{k-1}a_f^*(2\text{det}(T)/d^2).
$$
According to Andrianov \cite{An}, this condition can be equivalently replaced by the
Maass relation
$$
a(f;T)=\sum_{\substack{d\in\mathbb{N} \\ d\mid\varepsilon (T)}}
              d^{k-1}a\left(f;\begin{pmatrix} 1 & t/2d \\ \overline{t}/2d & mn/d^2\end{pmatrix}\right).
$$
We set $\mathcal{M}(k,\mathbb{H};R)=\mathcal{M}(k,\mathbb{H})\cap [\Gamma_2(\mathcal{O}),k;R]$
and $\mathcal{M}(k,\mathbb{H};R)_0=\mathcal{M}(k,\mathbb{H})\cap [\Gamma_2(\mathcal{O}),k;R]_0$.

A typical example of an element in $\mathcal{M}(k,\mathbb{H})$ is the (quaternionic) Eisenstein
series $E_{k,\mathbb{H}}$\;$(k>6,\,\text{even})$.
\begin{Thm} (Krieg \cite[Theorem 3]{K3})
\label{FC}
 {\it
\quad Let $k>6$ be an even integer. The Eisenstein series $E_{k,\mathbb{H}}$ is an element of
$\mathcal{M}(k,\mathbb{H};\mathbb{Q})$ and has the Fourier expansion
$$
E_{k,\mathbb{H}}(Z)=1+\sum_{\substack{T\in{\rm Sym}^\tau (2;\mathcal{O}) \\
T\geqq 0,T\ne 0}}a(E_{k,\mathbb{H}};T)e^{2\pi i\tau (T,Z)},
$$
where
$$
a(E_{k,\mathbb{H}};T)=\sum_{\substack{d\in\mathbb{N} \\ d\mid\varepsilon (T)}}
                               d^{k-1}\,a_k^*(2{\rm det}(T)/d^2)
$$
for $0\ne T\in{\rm Sym}^\tau (2,\mathcal{O})$,\;$T\geq 0$, and
$$
a_k^*(\ell)=
\begin{cases} \displaystyle
-\frac{2k}{B_k}    & \text{if\; $\ell=0$}\\
\displaystyle
-\frac{4k(k-2)}{(2^{k-2}-1)\,B_k\,B_{k-2}}[\sigma_{k-3}(\ell)-2^{k-2}\sigma_{k-3}(\ell/4)]
                       & \text{if\; $\ell\in\mathbb{N}$}.
\end{cases}
$$
Here, $B_m$ is the $m$-th Bernoulli number and
$$
\sigma_m(\ell)=
\begin{cases}
\sum_{d\mid \ell}d^m     & \text{if\; $\ell\in\mathbb{N}$}, \\
0                              & \text{otherwise}.
\end{cases}
$$
}
\end{Thm}
\begin{Rem}
\label{R1}
The condition ``$k>6$'' is necessary for the convergence of the Eisenstein series.
However, it is known that the above formula for the Fourier coefficients holds when $k$
is 4 or 6. This is justified by the so-called Hecke sum. See aslo \cite[Remark 1]{K3}.
\end{Rem}
Let $E_k$ denote the elliptic Eisenstein series of weight $k$ whose Fourier expansion
is
$$
E_k(z)=1-\frac{2k}{B_k}\sum_{n=1}^\infty\sigma_{k-1}(n)e^{2\pi iz},\quad z\in\mathbb{H}_1,
$$
where $\mathbb{H}_1$ is the upper-half plane.

For the Siegel operator $\Phi$, we obtain
$$
\Phi (E_{k,\mathbb{H}})=E_k
$$
when $k\geq 4$ is even.
\vspace{4mm}
\\
Herein, we introduce examples of cusp forms.

We set
\begin{align}
& X_{10}:=\frac{17}{161280}\left(E_{4,\mathbb{H}}E_{6,\mathbb{H}}-E_{10,\mathbb{H}}\right),
\label{x10}
\\
&X_{12}:=\frac{21421}{203212800}
            \left(\frac{441}{691}E_{4,\mathbb{H}}^3+\frac{250}{691}E_{6,\mathbb{H}}^2-E_{12,\mathbb{H}}\right).
            \label{x12}
\end{align}
Herein, these cusp forms are normalized as
$$
a\left(X_{10};\begin{pmatrix}1 & \tfrac{e_1+e_2}{2} \\ \tfrac{e_1-e_2}{2} & 1\end{pmatrix}\right)
=
a\left(X_{12};\begin{pmatrix}1 & \tfrac{e_1+e_2}{2} \\ \tfrac{e_1-e_2}{2} & 1\end{pmatrix}\right)=1
$$
and
$$
X_k\in \mathcal{M}(k,\mathbb{H};\mathbb{Z})_0\qquad (k=10,\,12).
$$
These forms are part of the set of generators of the graded ring of the quaternionic modular
forms of degree 2 constructed by Krieg \cite{K4}.
\section{Congruence properties of quaternionic modular forms}
\label{section3}
In this section, we provide some results regarding the congruence properties of quaternionic
modular forms of degree 2.

To study the congruence properties of the Eisenstein series, we consider a constant
multiple of $E_{k,\mathbb{H}}$:
\begin{equation}
\label{EisG}
G_{k,\mathbb{H}}:=-\frac{(2^{k-2}-1)\,B_k\,B_{k-2}}{4k(k-2)}\,E_{k,\mathbb{H}}.
\end{equation}
The following result is a simple consequence of Theorem \ref{FC} and Remark \ref{R1}.
\begin{Prop}
\label{FourierCoefG}
{\it Assume that $k\geq 4$ is even.
\vspace{2mm}
\\
{\rm (1)}\quad We thus assume that $T\ne O_2$. Then,
$$
a(G_{k,\mathbb{H}};T)=\sum_{d\mid\varepsilon (T)}d^{k-1}b_k^*(2{\rm det}(T)/d^2)
$$
and
$$
b_k^*(\ell)=
\begin{cases} \displaystyle
\frac{(2^{k-2}-1)\,B_{k-2}}{2(k-2)}    & \text{if\; $\ell=0$}\\
\sigma_{k-3}(\ell)-2^{k-2}\sigma_{k-3}(\ell/4)
                       & \text{if\; $\ell\in\mathbb{N}$}
\end{cases}
$$
{\rm (2)}\quad $\displaystyle a(G_{k,\mathbb{H}};O_2)=-\frac{(2^{k-2}-1)\,B_k\,B_{k-2}}{4k(k-2)}$.
}
\end{Prop}
\subsection{Ramanujan-type congruence}
Let
$$
\Delta (z)=q\,\prod_{n\geq 1}(1-q^n)^{24}\qquad (q=e^{2\pi iz})
$$
be the delta function, which is an elliptic cusp form of weight 12. The Ramanujan tau function
$\tau (n)$\;$(n\in\mathbb{N})$ is defined by
$$
\Delta (z)=\sum_{n=1}^\infty \tau (n)q^n.
$$
Ramanujan's congruence
$$
\tau (n) \equiv \sigma_{11}(n) \pmod{691}
$$
is interpreted as a congruence relation between the weight 12 elliptic Eisenstein series $G_{12}$ and $\Delta$ as follows:
\begin{equation}
\label{Ram}
G_{12} \equiv \Delta \pmod{691},
\end{equation}
where $G_k$ is defined as
$$
G_k:=-\frac{B_k}{2k}\,E_k.
$$

The following provides some results, which are the quaternionic version of (\ref{Ram}).
\begin{Thm}
\label{Ram} {\it 
Let $G_{k,\mathbb{H}}$ be the Eisenstein series defined in {\rm (\ref{EisG})}.
If a prime number $p\geq 5$ satisfies the conditions
$$
(*)\qquad\qquad
{\rm ord}_p\left(\frac{(2^{k-2}-1)\,B_{k-2}}{k-2}   \right)>0,\quad
{\rm ord}_p\left(\frac{B_k}{k}\right)\geq 0,
$$
then there is a quaternionic cusp form
$$
\chi_k\in [\Gamma_2(\mathcal{O}),k;\mathbb{Z}_{(p)}]_0
$$
such that
$$
G_{k,\mathbb{H}} \equiv \chi_k \pmod{p},
$$
where $\mathbb{Z}_{(p)}$ denotes the ring of $p$-integral rational numbers.
}
\end{Thm}
\begin{proof}
We apply the Siegel $\Phi$-operator to $G_{k,\mathbb{H}}$. Based on the assumption $(*)$
and Proposition \ref{FourierCoefG}, $\Phi (G_{k,\mathbb{H}})$ is an elliptic
modular form of weight $k$, whose Fourier coefficients are all divisible by $p$.
Hence, we can write 
$$
\Phi (G_{k,\mathbb{H}})=p\cdot f,
$$
where $f$ is an elliptic modular form of weight $k$ with $p$-integral Fourier
coefficients. Because $f$ has $p$-integral Fourier coefficients, we can write 
$$
f=P(E_4,E_6)\quad \text{with}\quad P(X_1,X_2)\in\mathbb{Z}_{(p)}[X_1,X_2].
$$
We then set
$$
F:=P(E_{4,\mathbb{H}},E_{6,\mathbb{H}}).
$$
Herein, we note that both $E_{4,\mathbb{H}}$ and $E_{6,\mathbb{H}}$ have $p$-integral
Fourier coefficients and $\Phi (E_{k,\mathbb{H}})=E_k\,(k=4,\,6)$.

If we set
$$
\chi_k:=G_{k,\mathbb{H}}-p\cdot F,
$$
by the above construction, we then obtain $\Phi (\chi_k)=0$ and $G_{k,\mathbb{H}} \equiv \chi_k \pmod{p}$.
\end{proof}
\begin{Ex}
(1)\quad When $k = 10$, we can take $p = 17$ as a prime number $p$ that satisfies the
condition $(*)$ and $\chi_{10}=X_{10}$, where $X_{10}$ is defined in (\ref{x10}).
We then obtain
$$
G_{10,\mathbb{H}} \equiv X_{10} \pmod{17}.
$$
As numerical examples,
{\footnotesize
\begin{align*}
1=
a\left(G_{10,\mathbb{H}};\begin{pmatrix} 1 & \tfrac{e_1+e_2}{2}\\ \tfrac{e_1-e_2}{2} & 1 \end{pmatrix}
            \right) & \equiv
  a\left(X_{10};\begin{pmatrix} 1 & \tfrac{e_1+e_2}{2} \\ \tfrac{e_1-e_2}{2} & 1\end{pmatrix}\right)
=1 \pmod{17}\\
129=
a\left(G_{10,\mathbb{H}};\begin{pmatrix} 1 & 0 \\ 0 & 1 \end{pmatrix}
            \right) & \equiv
  a\left(X_{10};\begin{pmatrix} 1 & 0 \\ 0 & 1\end{pmatrix}\right)
=-24 \pmod{17}\\
2188=
a\left(G_{10,\mathbb{H}};\begin{pmatrix} 1 & \tfrac{e_1+e_2}{2}\\ \tfrac{e_1-e_2}{2} & 2 \end{pmatrix}
            \right) & \equiv
  a\left(X_{10};\begin{pmatrix} 1 & \tfrac{e_1+e_2}{2} \\ \tfrac{e_1-e_2}{2} & 2\end{pmatrix}\right)
=12 \pmod{17}\\
{}                 &  \vdots \\
\end{align*}
}
(2)\quad When $k = 14$, we can take $p = 691$. We then set
\begin{equation}
\label{x14}
X_{14}:=E_{4,\mathbb{H}}X_{10}.
\end{equation}
We thus obtain
\begin{equation}
\label{Qanalogy}
G_{14,\mathbb{H}} \equiv X_{14} \pmod{691}.
\end{equation}
As numerical examples,
{\footnotesize
\begin{align*}
1=
a\left(G_{14,\mathbb{H}};\begin{pmatrix} 1 & \tfrac{e_1+e_2}{2}\\ \tfrac{e_1-e_2}{2} & 1 \end{pmatrix}
            \right) & \equiv
  a\left(X_{14};\begin{pmatrix} 1 & \tfrac{e_1+e_2}{2} \\ \tfrac{e_1-e_2}{2} & 1\end{pmatrix}\right)
=1 \pmod{691}\\
2049=
a\left(G_{14,\mathbb{H}};\begin{pmatrix} 1 & 0 \\ 0 & 1 \end{pmatrix}
            \right) & \equiv
  a\left(X_{14};\begin{pmatrix} 1 & 0 \\ 0 & 1\end{pmatrix}\right)
=-24 \pmod{691}\\
177148=
a\left(G_{14,\mathbb{H}};\begin{pmatrix} 1 & \tfrac{e_1+e_2}{2}\\ \tfrac{e_1-e_2}{2} & 2 \end{pmatrix}
            \right) & \equiv
  a\left(X_{14};\begin{pmatrix} 1 & \tfrac{e_1+e_2}{2} \\ \tfrac{e_1-e_2}{2} & 2\end{pmatrix}\right)
=252 \pmod{691}\\
{}                 &  \vdots \\
\end{align*}
}
As we will show later (Lemma \ref{FourierX14}), the Fourier coefficient $a(X_{14};T)$ is given as
\begin{align*}
& a(X_{14};T)=\sum_{d\mid\varepsilon (T)}d^{13}\tau^*(2\text{det}(T)/d^2),\\
& \tau^*(\ell)=\tau(\ell)-2^{12}\tau (\ell/4),
\end{align*}
if $\text{rank}(T) = 2$. By contrast,
\begin{align*}
& a(G_{14,\mathbb{H}};T)=\sum_{d\mid\varepsilon (T)}d^{13}\,b_{14}^*(2\text{det}(T)/d^2),\\
& b_{14}^*(\ell)=\sigma_{11}(\ell)-2^{12}\sigma_{11}(\ell/4).
\end{align*}
Therefore, the fact (\ref{Qanalogy}) reproduces
the congruence relation $\tau (n)\equiv \sigma_{11}(n) \pmod{691}$.
\vspace{2mm}
\\
(3)\quad Table 1 lists the prime numbers $p$ that satisfy the condition $(*)$.
\begin{table}[htbp]
\caption{Primes that satisfy $(*)$}
\begin{center}
\begin{tabular}{c||l|l|l|l|l|l|l|l|l|}
$k$  & $4$ & $6$ & $8$ & $10$ & $12$ & $14$ &     $16$      & $18$ &           $20$ \\ \hline
$p$  & {}     &  {}    &     {} & $17$ & $31$ &$691$&  $43,\,127$&$257,\,3617$ & $73,\,43867$     
\end{tabular}
\end{center}
\end{table}
\end{Ex}
\subsection{Theta operator}
For a Fourier series $F=\sum_T a(F;T)e^{2\pi i\tau(T,Z)}$, we formally set
$$
\Theta (F)(Z)=\sum_T a(F;T)\cdot (2\text{det}(T))e^{2\pi i\tau (T,Z)}.
$$
The Fourier series $\Theta(F)$ is not a modular form, even if $F$ is
such a form. We call this a {\it theta operator}.

In the cases of Siegel and Hermitian modular forms, we have some
results on the congruences arising from the theta operator (cf. \cite{B-K-N}, \cite{K-N-2}).

First, we prepare the following result.
\begin{Lem}
\label{weightp-1}{\it 
Let $p\geq 5$ be a prime number satisfying
$$
B_{p-3} \not\equiv 0 \pmod{p}.
$$
Then,
$$
E_{p-1,\mathbb{H}} \equiv 1 \pmod{p}.
$$
}
\end{Lem}
\begin{proof}
It is sufficient to show that $a(E_{p-1,\mathbb{H}};T) \equiv 0 \pmod{p}$ for $T$ with $\text{rank}(T)\geq 1$.
If $\text{rank}(T)=2$, we have
$$
a(E_{p-1,\mathbb{H}};T)=\frac{(p-1)(p-3)}{(2^{p-3}-1)\,B_{p-1}B_{p-3}}\cdot (\text{an integer}).
$$
Based on von Staudt-Clausen's theorem, 
$\displaystyle \text{ord}_p\left(\tfrac{p-1}{B_{p-1}}\right)>0$. By contrast, based on the
assumption ``$B_{p-3}\not\equiv 0 \pmod{p}$'' and the fact that $2^{p-3} -1 \not\equiv 0 \pmod{p}$, we have
$$
\text{ord}_p\left(\frac{p-3}{(2^{p-3}-1)B_{p-3}}\right)\geq 0.
$$
From these estimates, we obtain $a(E_{p-1,\mathbb{H}};T) \equiv 0 \pmod{p}$.

When $\text{rank}(T)=1$, we have
$$
a(E_{p-1,\mathbb{H}};T)=\frac{p-1}{B_{p-1}}\cdot (\text{an integer}).
$$
Again, using von Staudt-Clausen's theorem, we obtain $a(E_{p-1,\mathbb{H}};T) \equiv 0 \pmod{p}$.
Consequently,
$$
E_{p-1,\mathbb{H}} \equiv 1 \pmod{p}.
$$
\end{proof}
\begin{Rem}
(1)\quad In the theory of mod $p$ modular forms, the modular form $F_{p-1}$
of weight $p-1$ satisfying
$$
F_{p-1} \equiv 1 \pmod{p}
$$
plays an important role in determining the ring structure (cf. \cite{B-N},\,\cite{K-N-M2}, and \cite{K-N-M1}).

The above lemma asserts that such an $F_{p-1}$ can be constructed if the
prime number $p\geq 5$ satisfies the condition $B_{p-1}\not\equiv 0 \pmod{p}$
in the case of quaternion modular forms of degree 2. It is expected that
the existence of such $F_{p-1}$ will be shown without any conditions for $p$
in the case of quaternionic modular forms.
\\
(2)\quad Regarding the condition $B_{p-3}\not\equiv 0\pmod{p}$, we can find the same one in 
\cite[Theorem 2,1,(2)]{Na}. To the best of the author's knowledge, there are two primes $p$
satisfying $B_{p-3} \equiv 0 \pmod{p}$:
$$
p=16843,\quad 2124679\qquad  {\rm (cf. \,\, \cite{Na2}).  }
$$
Of course, a {\it regular prime} $p$ satisfies $B_{p-3}\not\equiv 0\pmod{p}$.
\end{Rem}
The second main result is as follows.
\begin{Thm}
\label{Theta}{\it 
Let $p\geq 5$ be a prime number satisfying $B_{p-3}\not\equiv 0\pmod{p}$.
Then, for any $F\in [\Gamma_2(\mathcal{O}),k;\mathbb{Z}_{(p)}]$, there is a
cusp form 
$$X\in [\Gamma_2(\mathcal{O}),k+p+1;\mathbb{Z}_{(p)}]_0
$$ 
satisfying
$$
\Theta(F) \equiv X \pmod{p}.
$$
}
\end{Thm}
\begin{proof}
For a given modular form $F$, we consider the Rankin-Cohen blacket
$[F,E_{p-1,\mathbb{H}}]$, where $E_{p-1,\mathbb{H}}$ is the Eisenstein
series of weight $p-1$ considered in the previous lemma.
The blacket $[F,E_{p-1,\mathbb{H}}]$ represents a cusp
form of weight $k+(p-1)+2=k+p+1$ with $p$-integral Fourier coefficients.
Moreover, by Lemma \ref{weightp-1}, 
$$
[F,E_{p-1,\mathbb{H}}] \equiv \Theta(F) \pmod{p}.
$$
This means that we may take $X=[F,E_{p-1,\mathbb{H}}]$.
\end{proof}
\begin{Ex}
(1)\quad $\Theta (G_{4,\mathbb{H}}) \equiv X_{10} \pmod{5}$
\vspace{2mm}
\\
As numerical examples,
{\footnotesize
\begin{align*}
1=
a\left(\Theta (G_{4,\mathbb{H}});\begin{pmatrix} 1 & \tfrac{e_1+e_2}{2}\\ \tfrac{e_1-e_2}{2} & 1 \end{pmatrix}
            \right) & \equiv
  a\left(X_{10};\begin{pmatrix} 1 & \tfrac{e_1+e_2}{2} \\ \tfrac{e_1-e_2}{2} & 1\end{pmatrix}\right)
=1 \pmod{5}\\
6=
a\left(\Theta (G_{4,\mathbb{H}});\begin{pmatrix} 1 & 0 \\ 0 & 1 \end{pmatrix}
            \right) & \equiv
  a\left(X_{10};\begin{pmatrix} 1 & 0 \\ 0 & 1\end{pmatrix}\right)
=-24 \pmod{5}\\
12=
a\left(\Theta (G_{4,\mathbb{H}});\begin{pmatrix} 1 & \tfrac{e_1+e_2}{2}\\ \tfrac{e_1-e_2}{2} & 2 \end{pmatrix}
            \right) & \equiv
  a\left(X_{10};\begin{pmatrix} 1 & \tfrac{e_1+e_2}{2} \\ \tfrac{e_1-e_2}{2} & 2 \end{pmatrix}\right)
=12 \pmod{5}\\
{}                 &  \vdots \\
\end{align*}
}
(2)\quad $\Theta (G_{6,\mathbb{H}}) \equiv X_{14} \pmod{7}$
\vspace{2mm}
\\
As further numerical examples,
{\footnotesize
\begin{align*}
1=
a\left(\Theta (G_{6,\mathbb{H}});\begin{pmatrix} 1 & \tfrac{e_1+e_2}{2}\\ \tfrac{e_1-e_2}{2} & 1 \end{pmatrix}
            \right) & \equiv
  a\left(X_{14};\begin{pmatrix} 1 & \tfrac{e_1+e_2}{2} \\ \tfrac{e_1-e_2}{2} & 1\end{pmatrix}\right)
=1 \pmod{7}\\
18=
a\left(\Theta (G_{6,\mathbb{H}});\begin{pmatrix} 1 & 0 \\ 0 & 1 \end{pmatrix}
            \right) & \equiv
  a\left(X_{14};\begin{pmatrix} 1 & 0 \\ 0 & 1\end{pmatrix}\right)
=-24 \pmod{7}\\
84=
a\left(\Theta (G_{6,\mathbb{H}});\begin{pmatrix} 1 & \tfrac{e_1+e_2}{2}\\ \tfrac{e_1-e_2}{2} & 2 \end{pmatrix}
            \right) & \equiv
  a\left(X_{14};\begin{pmatrix} 1 & \tfrac{e_1+e_2}{2} \\ \tfrac{e_1-e_2}{2} & 2 \end{pmatrix}\right)
=252 \pmod{7}\\
{}                 &  \vdots \\
\end{align*}
}
\end{Ex}
\subsection{Congruences mod 23}
In the process of determining the ring structure of degree 2 Siegel modular forms,
Igusa constructed a Siegel cusp form $X_{35}$ of odd weight 35.
In \cite{K-K-N}, the congruence relation
$$
\Theta (X_{35}) \equiv 0 \pmod{23}
$$
was reported. This fact means that, if $\text{det}(T)\not\equiv 0 \pmod{23}$,
then the corresponding Fourier coefficient $a(X_{35};T)$ is divisible by $23$.
After that, such Siegel modular forms were found one after another. For example,
$$
\Theta(E_{12,\mathbb{S}}) \equiv \Theta(\vartheta_{\mathcal{L}}) \equiv 0 \pmod{23}
$$
was proved, where $E_{12,\mathbb{S}}$ is the weight 12 Eisenstein series for the Siegel modular
group of degree 2, and $\vartheta_{\mathcal{L}}$ is the Siegel theta series associated
with the Leech lattice $\mathcal{L}$, which is also a Siegel modular form of weight 12.

In this section, we provide some congruence relation mod $23$ concerning
quaternionic modular forms.

In (\ref{x14}), we consider a cusp form $X_{14}$ defined by $X_{14}=E_{4,\mathbb{H}}X_{10}$.
\begin{Lem} {\it 
\label{FourierX14}
The modular form $X_{14}$ is an element of the Maass space $\mathcal{M}(14;\mathbb{H})_0$,
and the Fourier coefficient $a(X_{14};T)$ is given as follows:
\begin{align}
& a(X_{14};T)=\sum_{d\mid\varepsilon (T)}d^{13}\tau^*(2{\rm det}(T)/d^2) \label{F1}\\
& \tau^*(\ell)=\tau(\ell)-2^{12}\tau(\ell/4), \nonumber
\end{align}
where $\tau (n)$ is the Ramanujan tau function.
\\
In particular, $X_{14}\in\mathcal{M}(14;\mathbb{H};\mathbb{Z})_0$.
}
\end{Lem}
\begin{proof}
In \cite[Proposition 3]{K3}, Krieg constructed an isomorphism
$$
\Omega :\quad \mathcal{M}(k,\mathbb{H})\;\longrightarrow\;\mathfrak{M}_{k-2},
$$
where $\mathfrak{M}_k$ is a subspace of $[\Gamma_0(4),k]$ defined by certain conditions
on Hecke operators (cf. \cite{K3}). Under this isomorphism, we see the following:
$$
\Omega (X_{14})=\Delta(z)-2^{12}\Delta(4z).
$$
If we translate this fact into Fourier coefficients, we obtain (\ref{F1}).
\end{proof}
The third main result is as follows.
\begin{Thm}
\label{mod23}{\it 
Let $X_{14}$ be the cusp form defined above. If a matrix $T$ satisfies $\chi_{-23}(2{\rm det}(T))=-1$,
then the corresponding Fourier coefficient $a(X_{14},T)$ satisfies
$$
a(X_{14};T) \equiv 0 \pmod{23},
$$
where $\chi_{-23}(n)=\left(\frac{-23}{n}\right)$ is the Kronecker symbol.
}
\end{Thm}
\begin{proof}
First, we note that the following fact regarding $\tau(n)$ holds:
\\
If a prime number $p$ satisfies $\left(\frac{p}{23}\right)=-1$, then
$$
\tau (p) \equiv 0 \pmod{23}.
$$
We will now start the proof. Based on the previous Lemma, $a(X_{14};T)$ is given as
$$
a(X_{14};T)=\sum_{d\mid\varepsilon (T)}d^{13}\tau^*(2\text{det}(T)/d^2).
$$
We show that $\tau^*(2\text{det}(T)/d^2) \equiv 0 \mod{23}$ for any $d$ with $d\mid\varepsilon (T)$.
  
The assumption $\chi_{-23}(2\text{det}(T))=-1$ implies $\chi_{-23}(2\text{det}(T)/d^2)=-1$, 
and thus it is sufficient to show the following:
\vspace{2mm}
\\
$(**)$\quad If $\ell\in\mathbb{N}$ satisfies $\chi_{-23}(\ell)=-1$, then $\tau^*(\ell)\equiv 0 \pmod{23}$.
\vspace{2mm}
\\
When $\ell=4\ell'$, $\chi_{-23}(\ell)=-1$ is equivalent to $\chi_{-23}(\ell')=-1$. Hence, the following should be proved:
\begin{equation}
\label{tauell}
\tau (\ell) \equiv 0 \pmod{23}.
\end{equation}
We take the prime decomposition of $\ell$: $\ell=\prod p_i^{e_i}$\;(where $p_i$ is prime). From the assumption that $\chi_{-23}(\ell)=-1$,
there is a factor $p_i^{e_i}$ such that
$$
\chi_{-23}(p_i)=-1\quad \text{and}\quad e_i:\text{odd}.
$$
In general, the following recursion formula holds for any prime $p$:
$$
\tau(p^{n+1})=\tau(p^n)\tau(p)-p^{11}\tau(p^{n-1})\quad (n\geq 1).
$$
This fact implies that, if $\tau(p)\equiv 0 \pmod{23}$, then $\tau(p^{\text{odd}})\equiv 0\pmod{23}$.
\\
We apply this fact to the above prime number $p_i$. Because
$\chi_{-23}(p_i)=\left(\frac{-23}{p_i}\right)=\left(\frac{p_i}{23}\right)=-1$, we have $\tau(p_i)\equiv 0\pmod{23}$,
and necessarily, $\tau(p_i^{e_i}) \equiv 0 \pmod{23}$. Consequently
$$
\tau (\ell)=\tau(p_i^{e_i})\,\tau\big{(}\prod_{j\ne i}p_j^{e_j}\big{)}\equiv 0 \pmod{23}.
$$
This shows (\ref{tauell}), and thus $(**)$, which completes the proof of Theorem \ref{mod23}.
\end{proof}
\begin{Cor}
$$
\Theta_{\chi_{-23}}(X_{14})-\Theta (X_{14}) \equiv 0 \pmod{23}
$$
where $\Theta_\chi$ is a theta operator with character $\chi$ defined by
$$
\Theta_\chi (F)=\sum_{T}a(F;T)\cdot(2\text{det}(T))\cdot\chi(2\text{det}(T)e^{2\pi i\tau (T,Z)}
$$
for $F=\sum_T a(F;T)e^{2\pi i\tau (T,Z)}$.
\end{Cor}
\begin{Ex}
(1)\;If $T=\begin{pmatrix} 1 & \tfrac{e_1+e_2}{2} \\ \tfrac{e_1-e_2}{2} & 3 \end{pmatrix}$,
then $\chi_{-23}(2\text{det}(T))=\chi_{-23}(5)=-1$ and
$$
a(X_{14};T)=4830=2\cdot 3\cdot 5\cdot 7\cdot 23 \equiv 0 \pmod{23}.
$$
(2)\quad If  $T=\begin{pmatrix} 6 & 3(e_1+e_2) \\ 3(e_1-e_2) & 18\end{pmatrix}
                    =6\cdot\begin{pmatrix} 1 & \tfrac{e_1+e_2}{2} \\ \tfrac{e_1-e_2}{2} & 3 \end{pmatrix} $, and thus
$$
\varepsilon (T)=6\quad \text{and}\quad \chi_{-23}(2\text{det}(T))=\chi_{-23}(2^2\cdot 3^2\cdot 5)=-1.
$$
Then,
\begin{align*}
a(X_{14};T) &= \sum_{d\mid 6}d^{13}\tau^*(2\text{det}(T)/d^2)\\
              &=\tau^*(2^2\cdot 3^2\cdot 5)+2^{13}\tau^*(3^2\cdot 5)+3^{13}\tau^*(2^2\cdot 5)+6^{13}\tau^*(5),
\end{align*}
and
\begin{align*}
& \tau^*(2^2\cdot 3^2\cdot 5)=\tau(2^2\cdot 3^2\cdot 5)-2^{12}\tau(3^2\cdot 5) \equiv 0 \pmod{23},\\
& \tau^*(3^2\cdot 5)=\tau(3^2\cdot 5)\equiv 0 \pmod{23},\\
& \tau^*(2^2\cdot 5)=\tau(2^2\cdot 5)-2^{12}\tau(5)\equiv 0 \pmod{23},\\
& \tau^*(5)=\tau(5) \equiv 0 \pmod{23}.
\end{align*}
Therefore, we obtain $a(X_{14};T) \equiv 0 \pmod{23}$.
\end{Ex}
\subsection{Congruences of Eisenstein series}
In $\S$ \ref{section3}, we defined the Eisenstein series $G_{k,\mathbb{H}}$ of weight $k$.
We provide some congruence relatiion of $G_{k,\mathbb{H}}$:
\begin{Thm}
\label{CongEis}{\it
Let $G_{k,\mathbb{H}}$ be the quaternionic Eisenstein series of weight $k$ defined in {\rm (\ref{EisG})}.
Assume that $p=2k-5$ is a prime number.

If $T\in \text{Sym}^\tau (\mathcal{O})$ satisfies $\chi_{-p}(2\text{det}(T))=-1$,
then the corresponding Fourier coefficient $a(G_{k,\mathbb{H}},T)$ is divisible by $p$, namely,
$$
a(G_{k,\mathbb{H}};T) \equiv 0 \pmod{p},
$$
where $\chi_{-p}$ is the Kronecker symbol.
}
\end{Thm}
\begin{proof}
We may assume that $T>0$. We recall the formula for $a(G_{k,\mathbb{H}};T)$:
\begin{align*}
& a(G_{k,\mathbb{H}};T)=\sum_{d\mid\varepsilon (T)}d^{k-1}b_k^*(2{\rm det}(T)/d^2)\\
& b_k^*(\ell)=\sigma_{k-3}(\ell)-2^{k-2}\sigma_{k-3}(\ell/4).
\end{align*}
We show that
\begin{equation}
\label{bcong1}
b_k^*(2{\rm det}(T)/d^2) \equiv 0 \pmod{p}.
\end{equation}
Since $\chi_{-p}(2\text{det}(T))=-1$ is equivalent to $\chi_{-p}(2\text{det}(T)/d^2)=-1$,
it is sufficient to show that, if $\chi_{-p}(\ell)=-1$, then
\begin{equation}
\label{bcong2}
b_k^*(\ell) \equiv 0 \pmod{p}.
\end{equation}
Since $\chi_{-p}(\ell')=\chi_{-p}(\ell)$ when $\ell=4\ell'$, the proof of (\ref{bcong2}) is reduced to showing
the following:
\begin{equation}
\label{bcong3}
\sigma_{k-3}(\ell)=\sigma_{\frac{p-1}{2}}(\ell) \equiv 0 \pmod{p}\quad \text{if}\quad
\chi_{-p}(\ell)=-1.
\end{equation}
Similar to the proof of Theorem \ref{mod23}, we take the prime decomposition
$\displaystyle \ell=\prod_{i=1}^np_i^{e_i}$. Then, by assumption, there is a factor $p_j$\;$(1\leq j\leq n)$
satisfying
$$
\chi_{-p}(p_j)=-1\quad\text{and}\quad e_j\;\text{:\,odd}.
$$
We obtain
$$
\sigma_{\frac{p-1}{2}}(\ell)=\sum_{d\\ell}d^{\frac{p-1}{2}}\equiv \sum_{d|\ell}\left(\frac{d}{p}\right)
=\sum_{d|\ell}\chi_{-p}(d)=\prod_{i=1}^n\sum_{k_i=0}^{e_i}\chi_{-p}(p_i^{k_i}) \pmod{p},
$$
by the Euler criterion. If we consider the above factor $p_j$, we have
$$
\sum_{k_j=0}^{e_j}\chi_{-p}(p_j^{k_j})=\sum_{k_j=0}^{e_j}(-1)^{k_j}=0.
$$
This proves (\ref{bcong3}) and completes the proof of Theorem \ref{CongEis}.
\end{proof}
\begin{Cor}{\it
If $T\in \text{Sym}^\tau (\mathcal{O})$ satisfies $\chi_{-p}(2{\rm det}(T))=-1$,
then
$$
a(G_{14,\mathbb{H}};T) \equiv 0 \pmod{23}.
$$
}
\end{Cor}
\noindent
Compare with Theorem \ref{mod23}.
\vspace{4mm}
\\
\noindent


\begin{flushleft}
Shoyu Nagaoka\\
Department of Mathematics\\
Yamato University\\
Suita\\
Osaka 564-0082\\
Japan
\end{flushleft}

\end{document}